\documentclass{amsart}
\usepackage[utf8]{inputenc} 
\usepackage[T1]{fontenc}
\usepackage{lmodern}
\usepackage{graphicx}
\usepackage[english]{babel}
\usepackage{wrapfig}
\usepackage{amsthm}
\usepackage{setspace}
\usepackage{amsmath}

\usepackage{amssymb}

\usepackage{comment}
\usepackage{caption}
\usepackage[a4paper]{geometry}  

\usepackage[bottom]{footmisc}
\usepackage[all]{xy}
\usepackage{float} 
\usepackage{hyperref}
\usepackage{setspace}
\newsavebox{\auteurbm}
  {\small\slshape%
  \savebox{\auteurbm}{\upshape\sffamily#1}%
  \begin{flushleft}}
  {\\[4pt]\usebox{\auteurbm}
  \end{flushleft}\normalsize\upshape}
  
  \usepackage{mathtools}
  \usepackage{dsfont}
  \usepackage{amsmath}
  \usepackage{amsthm}

  \theoremstyle{definition}
  \newtheorem{defn}{Definition}[section]
  
  \newtheorem*{defn*}{Definition}
 
  \theoremstyle{plain}
  \newtheorem{thm}{Theorem}[section]
  
  \newtheorem{prop}{Proposition}[section]
  \newtheorem*{prop*}{Proposition}
  \newtheorem{lem}{Lemma}[section]
  \newtheorem{cor}{Corollary}[section]
   \newtheorem*{cor*}{Corollary}
  \newtheorem*{theo*}{Theorem}
  \newtheorem*{thm*}{Theorem}
  
  \theoremstyle{remark}
  
  \newtheorem{rem}{Remark}[section]

  \newtheorem{nota}{Notation}[section]

 \usepackage[toc,page]{appendix}
\usepackage{calrsfs}
\usepackage{tikz-cd}

\newcommand{\R}{\mathbb{R}}
\newcommand{\A}{\mathcal{A}}
\newcommand{\B}{\mathcal{B}}
\newcommand{\C}{\mathcal{C}}

\DeclareMathOperator*{\F}{U}

\newcommand{\M}{\mathcal{M}}

\newcommand{\U}{\operatorname{\mathcal{U}}}

\usepackage{pstricks,pst-node,pst-tree}

\usepackage{bm}

\usepackage{relsize}

\newcommand\independent{\protect\mathpalette{\protect\independenT}{\perp}}
\def\independenT#1#2{\mathrel{\rlap{$#1#2$}\mkern2mu{#1#2}}}

 \usepackage{xcolor}

 \newcommand\N{\mathbb{N}}
 \newcommand\cat[1]{\textbf{#1}}
 
 \newcommand{\da}{\hat{a}^*}

\newcommand{\cheek}[1]{\overset{\neg\vee}{{#1}}}

\newcommand{\Pa}{\mathcal{P}}

\newcommand{\g}{\mathcal{F}}

\DeclareMathOperator*{\vect}{\textbf{Gr}}
\DeclareMathOperator*{\W}{U}

\newcommand{\colim}{\operatorname{colim}}

\usepackage{appendix}

\title{Intersection property and interaction decomposition}

\author{Grégoire Sergeant-Perthuis}

\usepackage{pdfpages} 

\usepackage[intoc]{nomencl}
\makenomenclature

\usepackage{pdfpages}

\newcommand{\Mod}{\cat{Mod}}

\newcommand{\im}{\operatorname{im}}

\makeatletter
\newtheorem*{rep@theorem}{\rep@title}
\newcommand{\newreptheorem}[2]{%
\newenvironment{rep#1}[1]{%
 \def\rep@title{#2 \ref{##1}}%
 \begin{rep@theorem}}%
 {\end{rep@theorem}}}
\makeatother

  \theoremstyle{plain}
\newreptheorem{thm}{Theorem}
\newreptheorem{prop}{Proposition}
\newreptheorem{cor}{Corollary}

\makeatletter
\newcommand{\mylabel}[2]{#2\def\@currentlabel{#2}\label{#1}}
\makeatother

\numberwithin{equation}{section}

\newcommand{\Vect}{\cat{Vect}}

\usepackage{multiaudience}
\SetNewAudience{long}

\usepackage{minitoc}  
 
\usepackage{changepage} 
 
\usepackage{chngcntr}
\usepackage{etoolbox}
 \makeatletter
   {\par}
\makeatother

\AtBeginEnvironment{subappendices}{%
\section*{Appendix}
\addcontentsline{toc}{section}{Appendices}
\counterwithin{figure}{section}
\counterwithin{table}{section}
}

\usepackage{cite}

\begin{document}

\begin{abstract}
The decomposition into interaction subspaces is a hierarchical decomposition of the spaces of cylindrical functions of a finite product space, also called factor spaces. It is an important construction in graphical models and a standard way to prove the Hammersley-Clifford theorem that relates Markov fields to Gibbs fields and plays a central role in Kellerer's result for the linearized marginal problem. We define an intersection of sum property, or simply intersection property, and show that it characterizes collections of vector subspaces over a poset that can be hierarchically decomposed into direct sums, giving therefore a general setting for such construction to hold. We will call this generalization the interaction decomposition. The intersection property is the Bayesian intersection property, introduced in \cite{GS1}, when specified to factor spaces which, under this new perspective on the interaction decomposition, appears to be a structure property. An application is the extension of the decomposition into interaction subspaces for any product of any set.

\smallskip
\noindent \textbf{MSC2020 subject classifications:} Primary 06F25; secondary 62H22 \\
\noindent \textbf{Keywords.} Decomposition into interaction subspaces, Ordered vector spaces

\end{abstract}
\maketitle 

\section{Introduction}
\subsection{Motivation}

For a finite set $I$ and a finite product space $E=\prod_{i\in I} E_i$, the factor subspaces (or factor spaces) $\F(a)=\R^{E_a}$, seen as a subspace of $\R^E$, with $a\subseteq I$ and $E_a=\prod_{i\in a} E_i$, can be decomposed into direct sums of a collection of subspaces $(S_a\subseteq \R^E,a\subseteq I)$, called interaction subspaces (\cite{Kellerer1964,Lauritzen,Speed,Yeung,Matus}),

\begin{equation}\label{chapitre-2-decomposition-des-interactions}
\forall a\subseteq I \ \F(a)=\bigoplus_{\substack{b\in \Pa(I) \\b\subseteq a}} S_b
\end{equation}

In \emph{A Note on Nearest-Neighbour Gibbs and Markov Probabilities} \cite{Speed}, Speed traces back the first appearance of such a decomposition to Kellerer \cite{Kellerer1964}, Asmussen, Davidson \cite{Davidson} and Haberman \cite{Haberman}.\\

$\Pa(I)$, the power set of $I$, is a partially ordered set, or poset, and it can also be seen as a category with only one morphism between each of its elements $b\to a$, every times that $b\subseteq a$. In particular one can see $\F$ as a functor from a poset $\A=\Pa(I)$ to the category of vector spaces $\Vect$. In this document we characterize functors for which such decomposition exists.\\

Firstly we must give a definition of what a decomposition for a functor over a poset would be.

\begin{defn*}[Decomposable functor]
A functor $G:\A\to \Vect$ is decomposable if and only if there is a collection of vector spaces $(S_a,a\in \A)$ such that for any $a\in \A$ 

\begin{equation}
G(a)\cong \bigoplus_{b\leq a} S_a
\end{equation}

and for $b\in \A$ with $b\leq a$, $G^b_a$ is isomorphic to the inclusion $\bigoplus_{c\leq b} S_c\to \bigoplus_{c\leq a} S_c$. We will call $(S_a,a\in \A)$ a decomposition of $G$.

\end{defn*}

If a functor is decomposable then its morphisms $(G^b_a | a,b\in \A, b\leq a)$ are injective; we will call such functors as injective. We show (Proposition \ref{chapter-2-colim}) that every injective functor is isomorphic to an increasing functions from a poset $\A$ to the poset of vector subspaces of a vector space, seen as a category; we therefore restrict our attention to the latter for this article.

\begin{nota}

For $n\in \N$, we will note $[n]$ the integer interval $[0,n]$ and a collection of elements over a set $E$, $(f_x)_{x\in E}$, will also be referred to as $f_x, x\in E$.\\
In this document $\A$ will denote a poset and $V$ a vector space. 

\end{nota}

\begin{nota}
For $M$ a module we will denote $\vect M$ the poset of sub-modules of a module $M$; $\hom(\A,\vect M)$ will be the set of increasing functions between a poset $\A$ and $\vect M$; when $M$ is a vector space, $\vect M$ is also called the Grassmannian of $M$.

\end{nota}

\begin{rem}
As any poset is a category, in particular $\F\in \hom(\A,\vect V)$ is a functor from $\A$ to $\vect V$.
\end{rem}

\subsection{Main results of this document}

\subsubsection{Well-founded poset}
The main result of this document holds for well-founded poset, let us now recall what a well-founded poset is.

\begin{defn}\label{chapitre-2-well-founded_def}
A poset $\A$ is well-founded if any chain of $\A$ has a minimal element. This condition is often stated as the descending chain condition: every strictly decreasing sequences of elements of $\A$ terminates.
\end{defn}

\begin{rem}
Any non-empty subposet of a well-founded poset has at least one minimal element. 
\end{rem}

\begin{prop}\label{transfinite-induction}
Let $\A$ be a well-founded poset. To show that a property $P$ holds for any $a\in \A$ is suffices to show that,

\begin{equation}
\forall a\in \A, \quad [\forall b\in \A, b<a, \quad P(b)]\implies P(a)
\end{equation}

\end{prop}
\begin{proof}

Let $\A$ be well-founded. Let us assume that, 

\begin{equation}
\forall a\in \A, \quad [\forall b\in \A, b<a, \quad P(b)]\implies P(a).
\end{equation}

Let $\B$ be the set of elements of $\A$ that do not verify $P$, in other words,

\begin{equation}
\B=\{a\in \A |\quad \neg P(a)\}
\end{equation}

Let us suppose that $\B\neq \emptyset$. Let $b$ be a a minimal element of $\B$. For any $c<b$, $P(c)$ holds. But by hypothesis this implies that $P(b)$ holds. This is a contradiction. Therefore $\B$ is empty, which ends the proof. 
\end{proof}

\begin{rem}\label{induction-2-form}
Proposition \ref{transfinite-induction} is a generalization of the proof by induction and we will refer to it as the second form of proof by induction in this document.
\end{rem}

\subsubsection{Intersection property}

Let $\A$ be a poset, for any subposet $\B\subseteq \A$ of $\A$ one can define the lower-completion of $\B$, denoted as $\hat{\B}$, by,

\begin{equation}
\hat{\B}=\{ a\in \A: \ \exists b\in \B,  a\leq b\}
\end{equation}

For $a\in \A$ we denote $\hat{\{a\}}$ as $\hat{a}$. A lower set is a subposet of $\B\subseteq \A$ that is full under lower completion, i.e. $\hat{\B}= \B$; the set of lower sets of $\A$ will be denoted $\U(\A)$, it is the lower Alexandrov topology of $\A$.\\

The intrinsic condition that characterizes the decomposable functors is the following.

\begin{defn*}[Intersection property]
Let $\A$ be any poset, an increasing function $\W\in \hom(\A,\vect{V})$ is said to verify the intersection property $(I)$ if and only if, 

\begin{equation}\tag{I}
\forall \B,\C\in \U(\A), \quad \sum_{b\in \B}\W(b) \cap \sum_{c\in \C}\W(c)\subseteq \sum_{a\in \B\cap \C}\W(a)
\end{equation}

\end{defn*}

The intersection property in the context of generalized factor spaces is the Bayesian intersection property of \cite{GS1}. For well-founded poset there is a subset of conditions of $(I)$ that is sufficient to prove decomposability and that is therefore equivalent to $(I)$.

\begin{defn*}

$\W\in \hom(\A,\vect{V})$ is said to verify the property (C) if and only if,

\begin{equation} \tag{C}
\forall a\in \A, \sum_{b:b\leq a}\W(b) \cap \sum_{b:a\not \leq b}\W(b)\subseteq \sum_{b:b\lneqq a}\W(b)
\end{equation}
\end{defn*}

Showing that an increasing collection of vector subspaces satisfies $(C)$ is easier than showing $(I)$.

\subsubsection{Equivalence theorem}

\begin{repthm}{central-result-1}
Let $\A$ be a well-founded poset and $V$ any vector space. $\F\in \hom(\A,\vect{V})$ is decomposable if and only if $\F$ verifies (C) or (I).
\end{repthm}

\subsection{Structure of this document}

We choose to focus the first part of this document to proving the equivalence theorem for the simpler case of increasing function over a finite poset (Sections \ref{chapitre-2-section-1},\ref{chapitre-2-section-2}). We will then use this result to show this theorem in the general setting (Section \ref{chapitre-2-section-3}).\\

Being decomposable is not functorial, in other words let $\phi: \B\to \A$ be an increasing function and $\F\in \hom(\B, \vect V)$ then $\phi_{!} \F$ is in general not decomposable, where for any $a\in \A$,

\begin{equation}
\phi_{!}\F(a)= \sum_{b:\phi(b)\leq a} \F(b)
\end{equation}

In Section \ref{chapitre-2-section-4} we give some cases where the pushforward of a decomposable functor is decomposable. Finally we apply the results of \cite{GS1} and the equivalence theorem of this article to give a generalized version of the decomposition into interaction subspaces for factor spaces.

\section{Colimits of injective functors over a poset}

\begin{defn}[Injective functor]
A functor $G$ from a poset $\A$ to the category of $(R)$-modules $\cat{Mod}$ will be called injective if for any $a,b\in \A$ such that $b\leq a$, $G^b_a$ is a monomorphism, i.e. is injective.
\end{defn}

\begin{prop}\label{chapter-2-colim}
Let $G$ be an injective functor from any poset $\A$ to the category of $(R)$-modules $\cat{Mod}$. Let $(G_a,a\in \A, \colim G)$ be the initial cocone over $G$; then $G_a$ is injective for any $a\in \A$.

\end{prop}

\begin{proof}

Let $v\in G(a)$ be such that $[v_a]=0$; therefore there are $m\in \N$, $m_1\in \N$, a collection $(b_i\leq a,i\in [m])$ of elements of $\A$, a collection $(c_i\geq a,i\in [m_1]$, a collection $(u_i\in G(b_i), i\in [m])$, a collection $(w_i\in G(a), i\in [m_1])$ such that,

\begin{equation}
 v\times a =\sum_{i\in [m]} (G^{b_i}_a(u_i)\times a - u_i\times b_i) +\sum_{i\in [m_1]} (G^a_{c_i}(w_i)\times c_i - w_i\times a)
\end{equation}

Therefore there is a finite set $B$ of elements strictly less than $a$ and a finite set $C$ of elements strictly greater than $a$, there are two collections $(u_b\in G(b),b\in B)$ and $(w_c\in G(a),c\in C)$ such that,

\begin{equation}
v\times a=\sum_{b \in B}(G^b_a(u_b)\times a - u_b\times b)+ \sum_{c\in C}(G^a_c(w_c)\times c -w_c\times a
\end{equation}

As $B\cap C=\emptyset$, the projections on $b\in B$ and $c\in C$ gives $u_b=0$ and $w_c=0$. Therefore $v=0$.

\end{proof}

\begin{cor}
Let $G:\A\to \Mod$ be an injective functor, there is a module $M$ such that $G$ is isomorphic to $\tilde{G}\in \hom(\A, M)$
\end{cor}

\begin{proof}
Let $M= \colim_a G(a)$ and for any $a\in \A$ let $\tilde{G}(a)= \im G_a$ for any $a\in \A$, by Proposition \ref{chapter-2-colim}, $\tilde{G}$ is isomorphic to $G$.
\end{proof}

In what follows, without loss of generality, we consider functors in $\hom(\A, \vect(V))$.

\section{Decomposability and Intersection property}\label{chapitre-2-section-1}

\begin{defn}[Decomposable injective functor]\label{def-decomposition}
$S_a, a\in \A$ is a decomposition of $\W\in \hom(\A,\vect{V}) $ if and only if, 

\begin{enumerate}
\item for all $a\in \A$, $S_a \in \vect{V}$.
\item $\underset{a\in \A}{\sum}S_{a}$ is a direct sum in $V$; in other words,

\begin{center}
\begin{equation}\label{direc-sum-induction}
\begin{array}{ccccc}
p & : & \underset{a\in \A}{\bigoplus} S_a & \to & V \\
 & & (v_a)_a & \mapsto & \underset{a\in \A}{\sum} v_a \\
\end{array}
\end{equation}
\end{center}

is injective. 

\item for all $a\in \A$, $\W(a)=\underset{b\leq a}{\sum}S_b$.

\end{enumerate}

$\W$ is then said to be decomposable.
\end{defn}

\begin{rem} When $p$ in Equation (\ref{direc-sum-induction}) is injective, one says that $\underset{a\in \A}{\sum} S_a$ is in direct sum in $V$ and it is noted $\sum_{a\in \A} S_a=\bigoplus_{a\in \A} S_a$; with this notation the previous definition can be restated as $\underset{a\in \A}{\sum} S_a= \underset{a\in \A}{\bigoplus} S_a$ and $\forall a\in \A, \W(a)= \underset{b\leq a}{\bigoplus} S_b$.
\end{rem}

\begin{rem}
Let $V=\R^2$ and $a_1$, $a_2$, $a_3$, three lines in $V$ that are pairwise different. Let $\A=\{a_1,a_2, a_3\}$ and $\W:\A \overset{i}{\hookrightarrow} \vect{V}$ be the inclusion map. $\F$ is not decomposable: if it were, $\dim  \underset{a\in \A}{\sum} \W(a) =3$. 
\end{rem}

We would like to find a condition for which an increasing function $\F:\A\rightarrow \vect V$ is decomposable. The main result of this section is to show that $\F$ is decomposable if an only if it verifies a certain intersection of sums property.\\

\begin{defn}[Order-embedding]
 Let $\A$, $\B$ be two posets and let $f:\A\to \B$ be an increasing function. $f$ is an order-embedding if for all $a_1,a_1\in \B$, 

$$f(a_1)\leq f(a_2)\quad \implies \quad a_1\leq a_2$$

An order-embedding function is always injective. 
 \end{defn}
 
\begin{defn}[Intersection property]\label{intersection-property}
Let $\A$ be any poset, an increasing function $\W\in \hom(\A,\vect V)$ is said to verify the intersection property $(I)$ if and only if, 

\begin{equation}\label{def-intersection-property}\tag{I}
\forall \B,\C\in \U(\A), \quad \underset{b\in \B}{\sum}\W(b) \cap \underset{c\in \C}{\sum}\W(c)\subseteq \underset{a\in \B\cap \C}{\sum}\W(a)
\end{equation}

\end{defn}

Let $\B\subseteq \A$ be a subposet of $\A$, we will note $\W(\B)=\underset{b\in \B}{\sum}\W(b)$. Let us consider on $\mathcal{P}(\A)$ the following order,

\begin{equation}
\forall \B_{1},\B_{2}\in \mathcal{P}(\A) \quad \B_{1}\leq \B_{2} \quad \iff \quad \hat{\B}_{1}\subseteq \hat{\B}_{2}. 
\end{equation}

\begin{rem}
We have extended $\W$ to a poset morphism between the set of subposets of $\A$, $(\Pa(\A), \leq)$, to $\vect{V}$; in particular, if one calls $i$ the following application,

\begin{center}
$\begin{array}{ccccc}
i & : & \A & \to & \U(\A) \\
 & & a & \mapsto & \widehat{\{a\}} \\
\end{array}$
\end{center}

the extension $\left(i_{!}\W(\B)=\W(\B),\B\in \U(\A)\right)$ of $\W$ to $\U(\A)$ is the left adjoint of $i^*$, the inverse image of $\W$ by $i$.

\end{rem}

\begin{nota}
Several subposets of $\A$ will play an important role in the following development. Let us note,

\begin{eqnarray}\label{notations-1}
\hat{a}=\{b\in\A |b\leq a\} \\
\cheek{a}=\{b\in \A | a\not\leq b\} \nonumber\\
\da=\{b\in\A |b\lneqq a\} \nonumber
\end{eqnarray} 
\end{nota}

\begin{defn}\label{condition-1-def}

$\W\in \hom(\A,\vect{V})$ is said to verify the property (C) if and only if,

\begin{equation}\label{def-intersection-property-1} \tag{C}
\forall a\in \A, \W(a) \cap \W(\cheek{a})\subseteq \W(\da)
\end{equation}
\end{defn}

\begin{rem}
Let $V$ be a vector space and  $V_1$ a vector subspace of $V$, we will note the quotient map with respect to $V_1$, $\pi: V\rightarrow V/V_1$ as $\mod V_1$.\\

Let $V_1,V_2,V_3$ be three vector subspaces of $V$, $V_1$ is independent of $V_2$ conditionally to $V_3$, denoted as $V_1 \independent V_2 \vert V_3$ if and only if $(V_1 \mod V_3) \cap (V_2 \mod V_3)=0$. Condition $(C)$ can be rewritten in terms of conditional independence properties: 

\begin{equation}
\W(a) \independent \W(\cheek{a}) \big|\W(\da)
\end{equation}

\end{rem}

\begin{rem}
Condition $(C)$ is a subset of condition of $(I)$ and therefore $(I)$ implies $(C)$.
\end{rem}

Let us now show that decomposability implies a stronger version of intersection property.\\

\begin{defn}[Strong intersection property]\label{strong-intersection-property}
An increasing function $\F:\A\rightarrow \vect V$ is said to verify the strong intersection property $(sI)$ if and only if for any family  $(\A_j)_{j\in J}$ of elements of $\U(\A)$,

\begin{equation}\label{def-strong-intersection-property}\tag{sI}
\bigcap \limits_{j\in J} \F(\A_j)=\F(\bigcap \limits_{j\in J}\A_j).
\end{equation}

\end{defn}

\begin{rem} The motivation for this definition can be found in \cite{GS1} Theorem 4.1. It is a natural extension of the intersection property $(I)$; $(sI)$ implies $(I)$. Consider $\A=\N$ and let $V$ be a vector space, let for any $a\in \A$ $\F(a)=V$; $\F$ verifies $(I)$ but not $(sI)$ therefore condition $(sI)$ is in general strictly stronger than condition $(I)$.
\end{rem}

\begin{prop}\label{tentative}

Let $\F\in \hom(\A,\vect{V})$, with $\A$ any poset. If $\F$ is decomposable then $\F$ verifies $(sI)$.
\end{prop}
\begin{proof}

By hypothesis, for any $v\in \F(\A)$ there is a unique collection $s_a(v), a\in\A$, with $s_a(v)\in S_a$ such that $v=\underset{a\in \A}{\sum}s_a(v)$.\\

Let $(\A_j)_{j\in J}$ be a collection of elements of $\U(\A)$, and $v\in \bigcap \limits_{j\in J} \F(\A_j)$. \\

For any $j\in J$, $a\not \in \A_j$, $s_a(v)=0$. Therefore, for any $a \not \in \bigcap \limits_{j\in J}\A_j$, $s_a(v)=0$. Therefore $v= \underset{a\in \bigcap \limits_{j\in J}\A_j}{\sum} s_a(v)$.  \\

$\bigcap \limits_{j\in J} \A_j\in \U(\A)$ and $\F(\bigcap \limits_{j\in J} \A_j)=\underset{a\in \bigcap \limits_{j\in J}\A_j}{\sum} S_a$, so $v\in \F(\bigcap \limits_{j\in J} \A_j)$. The other inclusion is true whether or not $\F$ verifies $(sI)$.

\end{proof}

The rest of the document will be dedicated to proving that, under very general assumptions on the poset, if $\F$ satisfies the intersection property then it is decomposable.

\section{Intersection Property implies decomposability over finite posets}\label{chapitre-2-section-2}

\begin{lem}\label{butterfly-3}
Let $(\W(x), x\in E)$ be a any collection of vector subspaces of $V$ and $z\in E$.\\

$\underset{x\in E}{\sum} \W(a)\cong \underset{x\in E}{\bigoplus} \W(a)$ if and only if, for any $x\in  E$ such that $x\neq z$, 

\begin{equation}
\W(x)\cap \underset{\substack{y\in E \\ y\neq x}}{\sum} \W(y) = 0
\end{equation}

\end{lem}

\begin{proof}
Let $x_i, i\in [n]$ be a finite collection of elements of $X$.  Let, $v_i \in \W(a_i), i\in [n]$, $v_z\in \W(z)$, such that $\underset{i\in [n]}{\sum} v_i +v_z=0$. Then for any $i\in [n]$, $v_i=-\underset{j\neq i}{\sum}v_j$, so $v_i\in  \left(\underset{j\neq i}{\sum}\W(x_j)\right) \cap \W(x_i)$. Therefore by hypothesis for any $i\in [n]$, $v_i=0$. And so $ v_z=0$.\\
\end{proof}

\begin{thm}\label{central-result}

Let $\W\in \hom(\A,\vect{V})$, with $\A$ finite and $V$ any vector space.\\
 $\W$ is decomposable if and only if $\W$ verifies (C).
\end{thm}

\begin{proof}

The necessary condition is a direct consequence of Proposition \ref{tentative}.\\

\underline{Sufficient condition:} Let us prove inductively on the height of the posets that property (C) implies $\W$ to be decomposable. \\

If $h(\A)=0$, then $\A=\emptyset$. $\W=0$ is decomposable, by convention $\underset{i\in \emptyset}{\sum}V_i=0$.\\

Assume that for any poset of height lower than $n\in \mathbb{N}$, property (C) implies $\W$ decomposable. Let $\A$ be of height $n+1$. For any $a\in \A$, let $W_a=\W(\da)$\\

Let $\M$ be the set of maximal elements of $\A$.\\

For $a\in \M$, let $S_a$ be a supplementary of $W_a$, in other words, $\F(a)=S_a \oplus W_a$ and let $\B=\A\setminus \M$; $\B$ is a lower set.\\

Let us show that $\left( (S_a)_{a\in \M}, \F(\B) \right)$ verify the hypotheses of Lemma \ref{butterfly-3}.\\

Let $a\in \M$, $v\in S_a \cap \left(\underset{\substack{b\in \M \\ b\neq a}}{\sum}S_b + \F(\B)\right)$.\\

As $a$ is a maximal element in $\A$, $\cheek{a}= \A \setminus a$. In particular, $\underset{\substack{b\in \M \\ b\neq a}}{\sum}s_b + \F(\B)\subseteq \F(\A \setminus a)$, it is even an equality.\\

Therefore, $v\in \F(\da)$. However by construction, $S_a \cap \F(\da)=0$. So $v=0$. \\

Therefore by Lemma \ref{butterfly-3}, $\underset{a\in \M}{\sum} S_a + \F(\B)\simeq \underset{a\in \M}{\bigoplus} S_a \oplus \F(\B)$.\\

 Furthermore $\F(\A)= \underset{a\in \A}{\sum}\F(a) =\underset{a\in \M}{\sum} \F(a) + \F(\B)$. Let us recall that, for $a\in \M$, $\F(a)\subseteq S_a+ \F(\B)$. So  $\F(\A) \subseteq \underset{a\in \M}{\sum} S_a + \F(\B)$. The other inclusion is also true so, $\F(\A)= \underset{a\in \M}{\sum} S_a + \F(\B)$.\\

Therefore, $\F(\A)\cong \underset{a\in \M}{\bigoplus} S_a \oplus \F(\B)$. Furthermore, $\B$ is of height $n$ so $\F(\B)$ is decomposable. Therefore there is $S_b, b\in \B$ such that $\underset{b\in\B}{\sum}S_b\simeq \underset{b\in \B}{\bigoplus}S_b$ and for any $b\in \B$, $\F(b)=\underset{\substack{c\in \B \\ c\leq b}}{\sum} S_b$.\\

Hence $\underset{a\in\A}{\sum}S_a\cong \underset{a\in \A}{\bigoplus}S_a$.\\

Furthermore as $\B\in \U(\A)$ , $\F(b)=\underset{\substack{c\in \B \\ c\leq b}}{\sum} S_b=\underset{\substack{c\in \A \\ c\leq b}}{\sum}S_b$, which shows that $\F$ is decomposable and ends the induction. \\

\end{proof}

\begin{cor}\label{central-result-cor}
Let $\F\in \hom(\A,\vect{V})$, with $\A$ finite and $V$ any vector space. \\
$\F$ is decomposable if and only if $\F$ verifies (I).
\end{cor}
\begin{proof}
If $\F$ verifies (I) then it verifies (C), as (C) requires a subset of the conditions required for (I) to hold.
\end{proof}

\begin{cor}\label{condition 1} 
If $\A$ is finite, condition (I) and (C) are equivalent.
\end{cor}
\begin{proof}
We just saw that if $\F\in \hom(\A, \vect{V})$ verifies (I) it verifies (C). If $\F$ satisfies (C) then $\F$ is decomposable by Theorem \ref{central-result}. Therefore, by Proposition \ref{tentative}, $\F$ satisfies (C).
\end{proof}

\section{Decomposability for finite posets}\label{chapitre-2-section-3}

\subsection{Counter example and predecompositions}

\begin{rem}\label{counter-example-1}
Let $\A=\N$ and $V$ a vector space. For any $a\in \A$ let $\F(a)=V$. $\F$ verifies $(I)$. This example is a counter example to Theorem \ref{central-result} for non finite posets, in other words for infinite posets (I) does not necessarily imply $\F$ decomposable.
\end{rem}

 Let us first state what always holds when $\F$ has the intersection property.

\begin{defn}\label{pre-decomposition-def}
For $a\in \A$, $\pi_a :\F(a) \twoheadrightarrow \F(a)/\F(\da)$ is surjective. Let $s_a$ be a section of $\pi_a$ and let us also note $S_a$ the image of this section: $S_a=s_a(\F(a)/\F(\da))$. We call any such collection $S_a,a\in \A$ a pre-decomposition of $\F$.
\end{defn}

\begin{prop}
Let $\F\in \hom(\A,\vect{V})$ decomposable, and $s_a, a\in \A$ a decomposition of $\F$. $s_a, a\in \A$ is a pre-decomposition of $\F$. 
\end{prop}

\begin{proof}
Let $s_a,a\in \A$ be a decomposition of $\F$. Then for any $v\in \F(a)$, there is a unique $u\in s_a$ and $w\in \F(\da)$ such that $v=u+w$. Let $s_a(v)=u$.

For any $v\in \F(\da)$, $s_a(v)=0$. Therefore $s_a$ factorizes through $\pi_a$ and $s_a=s^{'}_a\circ \pi_a$. Furthermore for any $v\in \F(a)$, 

\begin{equation}
s_a^{'}(v \mod \F(\da) ) \mod \F(\da)=[s_a(v)]=[v]
\end{equation}

Therefore $s^{'}_a$ is a section of $\pi_a$, which ends the proof.
\end{proof}

\begin{rem}
$\F\in \hom(\A,\vect{V})$ is said to verify the property (C1) if and only if for any $n\in \N$, $a\in \A$ and finite collection $a_i,i\in [n]$ such that for any $i\in [n], a_i\in \cheek{a}$,

\begin{equation}\label{condition-2-prime}\tag{C1}
\F(a) \cap \F(\underset{i\in[n]}{\bigcup}\hat{a}_i)\subseteq \F(\da)
\end{equation}

By construction (C) and (C1) are equivalent. In practice one proves (C1) in order to show (C) which under certain assumptions on $\A$ (Corollary \ref{condition-1-prime-equiv}) implies (I).  
\end{rem}

\begin{prop}\label{injectivity}
Let $\F\in \hom(\A,\vect{V})$ and $(S_a, a\in \A)$ a pre-decomposition of $\F$. Let us suppose that $\F$ verifies (C), then, 

\begin{equation}
\underset{a\in \A}{\sum} S_a \cong \underset{a\in \A}{\bigoplus} S_a
\end{equation}

\end{prop}

\begin{proof}

Let us prove by induction on $n$ that for any collection $a_i,i\in [n]$ of elements of $\A$, $\underset{i\in I}{\sum}S_{a_i}\cong \underset{i\in I}{\bigoplus}S_{a_i}$. We will use the second form of proof by induction (see Remark \ref{induction-2-form}). \\

Let $n\in \N$, and suppose that for any $n'<n$ and any collection $a_i,i\in [n']$ of elements of $\A$, $\underset{i\in [n']}{\sum}S_{a_i}\cong \underset{i\in I}{\bigoplus}S_{a_i}$.\\ 

Let $a_i,i\in [n]$ be a collection of elements of $\A$.\\

$\{a_i | i\in [n]\}$ is a finite poset. Let $\M$ be the set of its maximal elements and respectively $J=a^{-1}(\M)$. Let $\overline{\M}=a([n]\setminus J)$.\\

Let $i\in J$. For any $j\neq i$, $a_j\in \cheek{a}_i$. Furthermore, any $b\in \overline{M}$ is in any $\cheek{a}_j$. Indeed, let $b\in \overline{M}$, then there is $j\in J$ such that $b\leq a_j$; if there is $k\in J$ such that $a_k\leq b$, then $a_k\leq a_j$ which is contradictory with $a_j$ being a maximal element.\\

Therefore, $\underset{\substack{j\in J\\ j\neq i}}{\sum} S_{a_j}+ \F(\overline{M}) \subseteq\underset{b\in \cheek{a}}{\sum}\F(b)$ and $S_{a_i}\cap \left(\underset{\substack{j\in J\\ j\neq i}}{\sum} S_{a_j}+ \F(\overline{M})\right)=0$.\\

Therefore by Lemma \ref{butterfly-3}, $\underset{i\in J}{\sum}S_{a_i}+ \F(\overline{M})\simeq \underset{i\in J}{\bigoplus}S_{a_i} \oplus \F(\overline{M})$.\\

$|[n]\setminus J|<n+1$ then by induction $\underset{i\in [n]\setminus J}{\sum}S_{a_i}\simeq \underset{i\in [n]\setminus J}{\bigoplus}S_{a_i}$.\\

Finaly, $\underset{i\in [n]\setminus J}{\sum}S_{a_i}\subseteq \F(\overline{M})$, therefore $\underset{i\in [n]}{\sum}S_{a_i}\simeq \underset{i\in [n]}{\bigoplus}S_{a_i}$. Which ends the proof.

\end{proof}

\begin{cor}\label{decomposable-predecomposition-decomposition}
$\F\in \hom(\A,\vect{V})$ is decomposable if and only if any pre-decomposition of $\F$ is a decomposition.\\

\end{cor}

\begin{rem}
If $\F$ has one pre-decomposition that is a decomposition, then any pre-decomposition is a decomposition.
\end{rem}

\begin{rem}
Corollary \ref{decomposable-predecomposition-decomposition} is in fact the easiest way to build decompositions and move decompositions around through increasing functions, as example, one can see the proof of Proposition \ref{decomposable-morphism}.
\end{rem}

We therefore need to find a condition on $\A$ for which any $\F$ that verifies (C) also verifies that for all $ a\in \A$, $\F(a)=\sum_{b\in \hat{a}} S_b$. A simple way to do so would be to reiterate the proof of Theorem \ref{central-result} but in a more general setting. We would like to have a stronger version of the proof by induction and one can do so if $\A$ is well-founded.  \\

\subsection{Well-founded posets}

A case of well-founded posets that we will often meet are graded posets. 

\begin{defn}
Let $\A$ be a poset. Assume that there is $r:\A\rightarrow \N$ a strictly increasing function, in other words,

\begin{equation}
\forall a,b\in \A, a<b \implies r(a)<r(b) 
\end{equation}

and that for any $a,b\in \A$, such that $a<b$ and such that there is no $c\in \A$ such that $a<c<b$, $r(b)=r(a)+1$.\\

Then one says that $\A$ is graded and $r$ is called the rank of $\A$.

\end{defn}

\begin{prop}\label{graded-well-founded}
Let $\A$ be a poset, $\B$ a well-founded posed and $r:\A\rightarrow \B$ a stricly increasing function. Then $\A$ is well-founded.\\

Therefore any graded poset is well-founded.

\end{prop}

\begin{proof}
Let $a_i,i\in \N$ be a strictly decreasing sequence of elements of $\A$. $r(a_i)$ is a strictly decreasing sequence sequence of elements of $\B$. Therefore there is $i\in \N$ such that, $\forall j\neq i$, $r(a_j)>r(a_i)$.\\

Therefore for any $j\in \N$, $j\neq i$, $\neg (a_j< a_i)$. As $a_i,i\in N$ is a chain it is a total order. Therefore for any $j\in \N$, $j\neq i$, $a_j\geq a_i$, and as $a_j\neq a_i$, one has that $a_j> a_i$. Therefore any $j> i$ is strictly greater than $i$, which is contradictory. \\

There is no infinite strictly descending chain in $\A$.\\

\end{proof}

\subsection{Main Theorem for Well-founded posets}

\begin{thm}\label{central-result-1}
Let $\A$ be a well-founded poset and $V$ any vector space. $\F\in \hom(\A,\vect{V})$ is decomposable if and only if $\F$ verifies (C).
\end{thm}

\begin{proof}

The necessary condition is a direct consequence of Proposition \ref{tentative}.\\

Let $\F\in \hom(\A,\vect{V})$, and $S_a,a\in\A$ be a pre-decomposition of $\F$.\\

Let us show by induction on $a\in \A$, that for any $a\in \A$, $\F(a)=\underset{b\leq a}{\sum} S_a$. 
We will use the second form of proof by induction (see Remark \ref{induction-2-form}). \\

Let us suppose that for any $b<a$, $\F(b)=\underset{c\leq b}{\sum} S_b$. \\

By construction, $\F(a)=S_a + \F(\da)$.\\

By hypothesis, for any $b< a$, $\F(b)=\underset{c\leq b}{\sum} S_c$. Therefore $\F(\da)=\underset{b<a}{\sum}\underset{c\leq b}{\sum}S_c \subseteq \underset{b<a}{\sum}S_b$. The other inclusion holds by definition of a pre-decomposition. \\

Therefore, $\F(a)=\underset{b\leq a}{\sum} S_a$.\\

As $\A$ is well-founded we can conclude (Proposition (\ref{transfinite-induction})) that for any $a\in \A$, $\F(a)=\underset{b\leq a}{\sum} S_a$.\\

Furthermore by Proposition (\ref{injectivity}), 

\begin{equation}
\underset{a\in \A}{\sum} S_a \cong \underset{a\in \A}{\bigoplus} S_a
\end{equation}

Therefore, $\F$ is decomposable. 

\end{proof}

\begin{cor}\label{central-result-cor-1}
Let $\F\in \hom(\A,\vect{V})$, with $\A$ well-founded and $V$ any vector space. \\

$\F$ is decomposable if and only if $\F$ verifies (I).
\end{cor}
\begin{proof}
If $\F$ verifies (I) then it verifies (C) as (C) requires a subset of the conditions required for (I) to hold.
\end{proof}

\begin{cor}\label{condition-1-prime-equiv} 
If $\A$ is well-founded, condition (I) and (C) are equivalent.
\end{cor}
\begin{proof}
We just saw that if $\F\in \hom(\A, \vect{V})$ verifies (I) it verifies (C). If $\F$ satisfies (C) then $\F$ is decomposable by Theorem \ref{central-result-1}. Therefore, by Proposition (\ref{tentative}), $\F$ satisfies (I).
\end{proof}

\section{A bit more around intersection and decomposability}\label{chapitre-2-section-4}

\subsection{A condition equivalent to (I)}
If one considers any poset $\A$ it is not clear whether condition (C) implies (I) however there is still a property easier to verify than property (I) that is equivalent to (I), for any poset $\A$.\\

\begin{defn}\label{condition-3-def}

Let $\A$ be any poset, $V$ any vector space. $\F\in \hom(\A,\vect{V})$ is said to verify the property (C2) if and only if for any $n\in \N$, $a\in \A$ and finite collection $(a_i,i\in [n])$,

\begin{equation}\label{condition-3}\tag{C2}
\F(a) \cap \F(\underset{i\in[n]}{\bigcup}\hat{a}_i)\subseteq \F(\hat{a}\cap \underset{i\in[n]}{\bigcup}\hat{a}_i)
\end{equation}

\end{defn}

\begin{prop}
For any poset $\A$ condition (I) and (C2) are equivalent.
\end{prop}

\begin{proof}
The proof of the sufficient condition is an adaptation of the one of Theorem 3.1 \cite{GS1} (Lemma 3.1 is in fact condition (C2)).\\

Let $\F\in \hom(\A,\vect{V})$.\\

Let $n_1\in \N$, $a_i, i\in [n_1]$ be a collection of elements of $\A$. Let us prove by induction on $n$ that for any $b_i, i\in[n]$ collection of elements of $\A$, 

\begin{equation}
\F(\underset{i\in [n_1]}{\bigcup}\hat{a}_i) \cap \F( \underset{i\in [n]}{\bigcup}\hat{b}_i)\subseteq \F(\underset{i\in [n_1]}{\bigcup}\hat{a}_i \cap  \underset{i\in [n]}{\bigcup}\hat{b}_i) 
\end{equation}

We will use the second form of proof by induction (see Remark \ref{induction-2-form}). \\

Let us assume that for any $m<n$, and any $c_i, i\in[m]$ collection of elements of $\A$, 

\begin{equation}
\F(\underset{i\in [n_1]}{\bigcup}\hat{a}_i) \cap \F( \underset{i\in [m]}{\bigcup}\hat{c}_i)\subseteq \F(\underset{i\in [n_1]}{\bigcup}\hat{a}_i \cap  \underset{i\in [m]}{\bigcup}\hat{c}_i) 
\end{equation}

\smallskip
Let $b_i, i \in [n]$ a collection of elements of $\A$.\\

Let $i\in [n]$, and $v\in \F(\underset{i\in [n]}{\bigcup}\hat{b}_i) \cap \F(\underset{i\in [n_1]}{\bigcup}\hat{a}_i) $, there is $v_i,i\in [n]$ such that for any $i\in [n], v_i\in \F(b_i)$ and $u_i,i\in [n_1]$ such for any $i\in [n_1], u_i\in \F(a_i)$, such that,

$$v= \underset{j\in [n]}{\sum} v_j=\underset{j\in [n_1]}{\sum}u_j$$

One has that, 

$$v_i = \underset{j\in [n_1]}{\sum}u_j -\underset{\substack{j\in [n]\\ j\neq i}}{\sum} v_j$$

Pose $\A_1=\underset{j\in[n_1]}{\bigcup}\hat{a}_j $, $\B_1=\underset{\substack{j\in [n]\\ j\neq i}}{\bigcup}\hat{b}_j$.\\

Therefore, $v_i\in \F(b_i)\cap \F(\A_1\cup B_1)$ and so $v_i\in \F(\hat{b}_i\cap (\A_1\cup B_1))$.\\

Therefore, $v_i\in \F((\hat{b}_i\cap \A_1)\cup (\hat{b}\cap B_1))=\F(\hat{b}_i\cap \A_1)+\F(\hat{b}_i\cap \B_1)$ and in particular there is $w_1\in \F(\hat{b}_i\cap \A_1)$, $w_2\in \F( \B_1)$ such that $v_i=w_1+w_2$.\\

Furthermore, $\underset{\substack{j\in [n]\\ j\neq i}}{\sum} v_j +w_2 = \underset{j\in [n_1]}{\sum}u_j -w_1$.\\

But $\underset{\substack{j\in [n]\\ j\neq i}}{\sum} v_j +w_2\in \F(\B_1)$ and $\underset{j\in [n_1]}{\sum}u_j -w_1\in \F(\A_1)$.\\

One also has that $|\{b_j |\quad j\neq i\}|<n$ then by the induction hypothesis $\underset{\substack{j\in [n]\\ j\neq i}}{\sum} v_j +w_2 \in \F(\A_1\cap \B_1)$. \\

Therefore $v\in \F(\underset{i\in [n]}{\bigcup}\hat{b}_i \cap \underset{i\in [n_1]}{\bigcup}\hat{a}_i) $, which ends the proof.

\end{proof}

\subsection{Fonctoriality of decomposability}

\begin{nota}
Let $\A$ be a poset and $\B\subseteq \A$. Let $\F$ be an increasing function from $\A$ to $\vect{V}$. Let us note $\F|_\B\in \hom(\B,\vect{V})$, the restriction of $\F$ to $\B$.\\
\end{nota}

\begin{prop}\label{decomposition-restriction-property}
Let $\A$ be a poset and $\B\in \U(\A)$, let $\F\in \hom(\A, \vect{V})$ decomposable, then $\F|_\B$ is decomposable. If $\F$ verifies (I) so does $\F|_\B$.

\end{prop}

\begin{proof}
Suppose that $\F$ is decomposable, let $(S_a,a\in \A)$ be a decomposition of $\F$, then for any $a\in \B$,

\begin{equation}
\F(a)= \bigoplus_{\substack{b\in \A \\ b\leq a } } S_a= \bigoplus_{\substack{b\in \B\\ b\leq a } } S_a
\end{equation}

Suppose that $\F$ satisfies (I). Any $\C\in \U(\B)$ is also in $\U(\A)$ therefore $\F|_\B$ is decomposable.
\end{proof}

\begin{prop}\label{decomposable-morphism}
Let $f:\A\to \B$ be an order-embedding and $\F\in \hom(\A,\vect{V})$.\\

$\F$ is decomposable if and only if $f_{!}\F$ is decomposable. \\

Furthermore, any decomposition of $f_{!}\F$ induces, by restriction to $f(\A)$, a decomposition of $\F$; in other words, if $S^{\B}_{b},b\in \B$  is a decomposition of $f_{!}\F$ then $S^{\A}_a=S^{\B}_{f(a)}, a\in \A$ is a decomposition of $\F$.\\

Any decomposition of $\F$ can be extended in a decomposition of $f_{!}\F$. Let $S^{\A}_a, a\in \A$ be a decomposition of $\F$ then,

\begin{equation}
S^{\B}_{b}= \left\{
    \begin{array}{ll}
        S^{\A}_a & \mbox{if } \exists a\in \A, \quad b=f(a) \\
        0 & \mbox{otherwise}
    \end{array}
\right.
\end{equation}

is a decomposition of $f_{!}\F$.

\end{prop}

\begin{proof}

Let $S^{\B}_{b},b\in \B$ be a pre-decomposition of $f_{!}\F$. As for any $a\in \A$, $f_{!}\F(f(a))=\F(a)$ and by definition, $\F(\da)=f_{*}\F(\widehat{f(a)}^{*})$, one has by construction that $S^{\A}_a=S^{\B}_{f(a)}, a\in \A$ is a pre-decomposition of $\F$. \\

\underline{Assume $\F$ is decomposable}, then (Corollary \ref{decomposable-predecomposition-decomposition}) for any $a\in \A$, $f_{!}\F(f(a))= \F(a)=\underset{a_1\in \hat{a}}{\sum}S^{\A}_{a_1}$ and $\F(f(a))=\underset{b_1\in f(\hat{a})}{\sum}S^{\B}_{b_1}$. One has that $\underset{b_1\in f(\hat{a})}{\sum}S^{\B}_{b_1}\subseteq \underset{b_1\in \widehat{f(a)}}{\sum}S^{\B}_{b_1}$. \\

Furthermore, for any $b\in \B$, $f_{!}\F(b)=\underset{a : i(a)\leq b}{\sum} \F(a)$. Then one has that, $f_{!}\F(b)\subseteq \underset{a : i(a)\leq b}{\sum}S^{\A}_a\subseteq  \underset{b_1\in  \hat{b}}{\sum}S^{\B}_{b_1}$. \\

As, by Proposition (\ref{injectivity}), $\underset{b\in \B}{\sum}S^{\B}_{b}=\underset{b\in\B}{\bigoplus}S^{\B}_{b}$, we can conclude that $f_{!}\F$ is decomposable.\\

\underline{If $f_{!}\F$ is decomposable}, then $S^{\B}_{b},b\in \B$ is a decomposition of $i_{!}\F$ (Corollary \ref{decomposable-predecomposition-decomposition}). One remarks that,

\begin{equation}\label{extension-order-embedding}
S^{\B}_{b}= \left\{
    \begin{array}{ll}
        S^{\A}_a & \mbox{if } \exists a\in \A, \quad b=f(a) \\
        0 & \mbox{otherwise}
    \end{array}
\right.
\end{equation}

Let $a\in \A$, $\F(a)=\underset{b \in \widehat{f(a)}}{\sum}S^{\B}_{b}$. Furthermore from Equation (\ref{extension-order-embedding}), one can conclude that $\underset{b\in \widehat{f(a)}}{\sum}S^{\B}_{b}=\underset{a_1\in \hat{a}}{\sum}S^{\A}_{a_1}$.\\

Therefore $\F(a)=\underset{a\in \hat{a}}{\sum}S^{\A}_{a}$ and by Proposition (\ref{injectivity}) one concludes that $\F$ is decomposable.
\end{proof}

\begin{cor}\label{decomposable-powerset}
$\F\in \hom(\A,\vect{V})$ is decomposable if and only if $i_{!}\F\in \hom(\U(\A),\vect{V})$ is decomposable.\\
\end{cor}

\begin{proof}
$i$ is an order-embedding, so one can conclude thanks to Proposition \ref{decomposable-morphism}.

\end{proof}

\begin{cor}\label{cor-extension-powerset}
Let $\F\in \hom(\A,\vect{V})$ be decomposable. Let $\B\subseteq \U(\A)$, such that $i(\A)\subseteq \B$. Then $i_{!}\F|_\B$ is decomposable.

\end{cor}

\begin{proof}
$i$ restricted on its codomain to $\B$ is an order-embedding.
\end{proof}

\section{Interaction decomposition for collections of random variables}\label{section-4}

Let us now show a generalization of the decomposition into interaction subspaces; to do so let us first recall some results from \cite{GS1}. Let $I$ denote a finite set. The poset $\A$ will be $\Pa(I)$. Let for all $i\in I$, $E_i$ be \textbf{any} non empty set. $E=\underset{i\in I}{\prod} E_i$ is a set of functions on $I$. For $x\in E$, one has that $pr_i(x)=x(i)$, and for $a\subseteq I$ non empty,  we will note $x_{|a}$ as $x_a$. We will call $E_a=\underset{i\in a}{\prod} E_i$ and,

$$\begin{array}{ccccc}
\pi_a& : &E & \to & E_a\\
& & x & \mapsto &x_a\\
\end{array}$$

Let $\ast$ be a given singleton. Then there is only one application of domain $E$ to $\ast$ that we call $\pi_\emptyset$; we pose $x_\emptyset=\pi_\emptyset (x)$. Pose $\g=\R^E$.\\

\begin{defn}[Factor subspaces]
For any $a\in \Pa(I)$, let $\F(a)$ be the vector subspace of $\g$ constituted of functions $f$ that can be factorised by $\pi_a$, in other words there is $\tilde{f}$ such that $f=\tilde{f}\circ \pi_a$. 

\end{defn}

$\F$ is an increasing function of $\hom(\Pa(I),\vect(\g))$, furthermore Theorem 3.1 \cite{GS1} asserts that for any $\B_1,\B_2\in \U(\Pa(I))$,

\begin{equation}\label{intersection-for-variables}
\F(\B_1)\cap \F(\B_2) \subseteq \F(\B_1 \cap \B_2)
\end{equation}

\begin{cor}[Decomposition of interactions]\label{decomposition-of-interactions}
$\F$ is decomposable.
\end{cor}
\begin{proof}

As the intersection property $(I)$ (Equation (\ref{intersection-for-variables})) holds for $\F$, by Theorem \ref{central-result}, $\F$ is decomposable.

\end{proof}

\begin{rem}
 We will call any choice of decomposition of $\F$ a decomposition of interactions. In Appendix B Proposition B.4 of \cite{Lauritzen} a decomposition is chosen with respect to the euclidian scalar product, and other ones can be found in \cite{Yeung}.

\end{rem}

Let now $I$ be any set. Let $E=\underset{i\in I}{\prod}E_i$ for any collection of sets $E_i, i\in I$. Let us note $\Pa_f(I)$ the set of subsets of $I$ which are finite and once again $\g=\R^E$.

\begin{cor}\label{decomposition-of-interactions-extended}
$\Pa_f(I)$ is well-founded.
\end{cor}

\begin{proof}
The cardinal function on $\Pa_f(I)$ to $\N$, $|.|:a\rightarrow |a|$ is a rank for $\Pa_f(I)$ and $\Pa_f(I)$ is graded. Therefore, Corollary \ref{graded-well-founded}, $\Pa_f(I)$ is well-founded. 
\end{proof}

For $a\in \Pa_f(I)$. $\F(a)$ will still be the set of functions from $E$ to $\R$ that factorise through $\pi_a$ and $\F\in \hom(\Pa_f(I),\vect(\g))$.\\

\begin{prop}
$\F$ is decomposable.
\end{prop}
\begin{proof}
Corollary 4.1 \cite{GS1} states that $\F$ verifies the intersection property $(I)$ therefore as $\Pa_f(I)$ is well-founded, Corollary \ref{central-result-cor-1} tells us that $\F$ is decomposable.
\end{proof}

Let us conclude this section by remarking that one can extend $\F$ to $\Pa(I)$ and that this extension is decomposable even though $\Pa(I)$ is not necessarily well-founded.\\

\begin{defn}(Generalized Factor subspaces)\\
One can always inject $\Pa(I)$ in $\U(\Pa_f(I))$ by the increasing function, 
\begin{center}
$\begin{array}{ccccc}
j & : & \Pa(I) & \to & \U(\Pa_f(I)) \\
 & & a & \mapsto & \{b\in \Pa_f(I) | b\leq a\} \\
\end{array}$

\end{center}

For $a\in \Pa(I)$, let $H(a)= i_{!}\F(j(a))$; $H\in \hom(\Pa(I),\vect(\g))$ (see Section 4 \cite{GS1}).
\end{defn}

\begin{rem}
$\hat{a}$ is not defined as $a\not \in \Pa_f(I)$. Furthermore $j$ is even an order-embedding
\end{rem}

\begin{prop}
$H$ is decomposable.
\end{prop}

\begin{proof}

Let us identify $\Pa(I)$ with $j(\Pa(I))$, then $H$ can be identified to $i_{*}\F_{|j(\Pa(I))}$. As $i(\Pa_f(I))\subseteq j(\Pa(I))$, by Corollary \ref{cor-extension-powerset}, $H$ is decomposable.

\end{proof}

\section*{Funding}

This work was supported by Université de Paris.

\bibliographystyle{ieeetr}
\bibliography{bintro}





\end{document}